\documentclass[12pt]{article}
\usepackage{a4wide}
\usepackage{amsmath,amssymb,amsthm,amsrefs}
\usepackage[mathscr]{eucal}
\usepackage[shortlabels]{enumitem}
\usepackage{tikz}
\usetikzlibrary{positioning,automata}
\usetikzlibrary{arrows,calc}

\author{Ievgen Bondarenko}
\title{\textbf{Quaternionic lattices and poly-context-free word problem}}

\newcommand{\WP}{\mathrm{WP}}
\newcommand{\bF}{\mathbb{F}}

\usepackage{eqparbox}

\newcommand\scalemath[2]{\scalebox{#1}{\mbox{\ensuremath{\displaystyle #2}}}}

\newtheorem{theorem}{Theorem}
\newtheorem{proposition}[theorem]{Proposition}
\newtheorem{corollary}{Corollary}[theorem]
\newtheorem{lemma}{Lemma}
\theoremstyle{definition}
\newtheorem{definition}{Definition}
\newtheorem{example}{Example}

\newtheorem{conjecture}{Conjecture}

\newtheorem{remark}{Remark}

\begin{document}
\maketitle

\begin{abstract}
A finitely generated group $G$ is called poly-context-free if its word problem $\WP(G)$ is an intersection of finitely many context-free languages. We consider the quaternionic lattices $\Gamma_\tau$ over the field $\mathbb{F}_{q}(t)$ constructed by Stix-Vdovina (2017), and prove that they are not poly-context-free. As a corollary, since all the groups $\Gamma_{\tau}$ are quasi-isometric to $F_2\times F_2$, the class of groups with poly-context-free word problem is not closed under quasi-isometries. The result follows from the description of the language $\WP(\Gamma_\tau)\cap a^*b^*c^*d^*$, which relies on the existence of anti-tori and certain power-type endomorphisms of the groups $\Gamma_\tau$.

\vspace{0.2cm}\textit{2020 Mathematics Subject Classification}: 20F10, 68Q45, 20F65, 20E08

\end{abstract}


\section{Introduction}

A famous result of Muller and Schupp \cite{Muller-Schupp} shows that a finitely generated group $G$ has context-free word problem if and only if $G$ is virtually free if and only if the Cayley graph of $G$ is quasi-isometric to a tree. Numerous generalizations of this result have emerged over the past decades (see \cite{HRR_book:GroupLangAut} and the references therein).

Intersection of context-free languages is not necessary context-free. A language is called poly-context-free if it is an intersection of finitely many context-free languages. The word problem is poly-context-free independently of the finite generating set. The class of poly-context-free groups is closed under taking finitely generated subgroups, finite index overgroups, and finite direct products. The natural examples of poly-context-free groups are finite direct products of free groups. Brough \cite{Brough} conjectured that a group is poly-context-free if and only if it is virtually a finitely generated subgroup of a direct product of free groups. She checked the conjecture for certain classes of groups: finitely generated nilpotent, polycyclic, and Baumslag-Solitar groups.

In this paper we study the word problem in the groups $\Gamma_{\tau}$ introduced by Stix-Vdovina \cite{StixVdovina:SimplyTrans}. The groups $\Gamma_{\tau}$ are arithmetic quaternionic lattices in positive characteristic and act simply transitively on the product of $(q+1)$-ary trees $T_{q+1}\times T_{q+1}$. The fact that they share the same geometry as the direct product of two free groups makes them interesting from the point of view of the aforementioned conjecture.

\begin{theorem}
The word problem in the groups $\Gamma_{\tau}$ is not poly-context-free.
\end{theorem}

We show this by describing the intersection $L=\WP\cap a^*b^*c^*d^*$ for certain relations $abcd=e$ in the group. We prove that $L$ is an infinite language of logarithmic growth and not poly-context-free.
The key property is the existence of certain power-type endomorphisms of the groups $\Gamma_\tau$ (see Proposition~\ref{prop:endomorphism_phi_k}). Our approach essentially use positive characteristic of $\Gamma_\tau$, and seems to not generalize to other lattices in the product of trees (see Examples~\ref{ex:group G_4},\ref{ex:group Wise}).

Since all the groups $\Gamma_{\tau}$ are quasi-isometric to $F_2\times F_2$, we get

\begin{corollary}
The class of groups with poly-context-free word problem is not closed under quasi-isometries.
\end{corollary}

\section{Semilinear sets and bounded context-free languages}

Let $\mathbb{N}_0$ denote the set of non-negative integers. A subset $L\subset \mathbb{N}_0^d$ is called \textit{linear} if there exist $c\in \mathbb{N}_0^d$ and a finite subset $P=\{p_1,\ldots,p_m\}\subset \mathbb{N}_0^d$ such that
\[
L=L(c,P)=\{c+\sum_{i=1}^m \lambda_ip_i : \lambda_i\in \mathbb{N}_0\}.
\]
A union of finitely many linear sets is called \textit{semilinear}. The union and intersection of finitely many linear or semilinear sets is semilinear (see \cite{Ginsburg}). The growth function $\gamma_L(n)$ of a subset $L$ is equal to the number of tuples $(n_1,\ldots,n_d)\in L$ such that $|n_i|\leq n$. The growth functions of semilinear sets are quasi-polynomial (see \cite{Quasi-polynomials}), and cannot have logarithmic behavior.

Let $\Sigma$ be a finite alphabet. Let $\Sigma^{*}$ denote the set of all words over $\Sigma$, including the empty word $\epsilon$. For $w\in \Sigma^{*}$ we write $w^{*}$ instead of $\{w\}^{*}$. A language $L\subseteq\Sigma^{*}$ is called \textit{bounded} if $L\subseteq w_1^{*}w_2^{*}\ldots w_d^{*}$ for some words $w_1,w_2,\ldots,w_d\in\Sigma^{*}$. In this case one can define
\[
P(L)=\{ (n_1,\ldots,n_d)\in\mathbb{N}_0^d : w_1^{n_1}\ldots w_d^{n_d}\in L\}.
\]
The well-known Parikh's theorem states that if a bounded language $L$ is context-free, then $P(L)$ is a semilinear set (see \cite{Ginsburg}). The intersection of a context-free language with a regular language is context-free. It follows that if $L$ is a poly-context-free language, then $P(L\cap w_1^{*}\ldots w_d^{*})$ is an intersection of finitely many semilinear sets, and therefore it is semilinear.

Let $G$ be a finitely generated group and $S$ a finite symmetric generating set, $S=S^{-1}$. The word problem of $G$ with respect to $S$ is the language
\[
\WP(G,S)=\{ w\in S^{*} : w=_Ge \}.
\]
Our strategy to prove that the word problem is not poly-context-free is to intersect it with a bounded language, describe the corresponding subset of $\mathbb{N}_0^d$, and show that it is not semilinear. In our case, we get bounded languages of logarithmic growth.

\section{The smallest example $\Gamma_3$}\label{sect:the_smallest_example}

In this section, we give an almost self-contained treatment of the ``smallest'' example of the family $\Gamma_\tau$, corresponding to the prime $p=3$. Consider the group
\[
\Gamma_3=\langle a,b,x,y\, |\, ax=x^{-1}b, ay=y^{-1}b^{-1}, ay^{-1}=xa^{-1}, bx=yb^{-1}\rangle.
\]
The group $\Gamma_3$ is a residually finite torsion-free just-infinite group; it is one of two irreducible lattices that appear as the fundamental group of a complete square complex with one vertex and four squares (see \cite{BK:AutomCSC}). The group $\Gamma_3$ admits the following presentation in the group $SL_3(\mathbb{F}_3(t))$:
\begin{align*}
\langle \scalemath{0.75}{
\left(  
  \begin{array}{ccc}
  -1-t & t^2-t & 0 \\
     1 & -t-1  & 0 \\
     0 & 0     & 1 \\
  \end{array}
\right),
\left( 
  \begin{array}{ccc}
  -1-t & 0 & t^2-t \\
     0 & 1 & 0 \\
     1 & 0 & -1-t \\
  \end{array}
\right),
\frac{1}{t+1}
\left(
  \begin{array}{ccc}
    -1 & t-t^2 & t^2-t \\
    -1 & -t    & 1-t \\
     1 & 1-t   & -t \\
  \end{array}
\right),
\frac{1}{t+1}
\left(
  \begin{array}{ccc}
    -1 & t-t^2 & t^2-t \\
    -1 & -t    & t-1 \\
    -1 & t-1   & -t \\
  \end{array}
\right) } \rangle.
\end{align*}
The unlabeled Cayley graph of $\Gamma_3$ with respect to $S=\{a^{\pm 1},b^{\pm 1},x^{\pm 1},y^{\pm 1}\}$ is isomorphic to the Cayley graph of $F_2\times F_2$.

\begin{proposition}
Let $P=\{ (i,j,k,l)\in\mathbb{Z}^4 : a^ix^j=x^{-k}b^l)\}$. Then
\begin{align*}
P&=\{(0,0,0,0)\}\cup\{(0,n,-n,0) : n\in\mathbb{Z}\}\cup \\
& \quad\cup\{ \pm( 3^n,-3^n,3^n,3^n):n\in 2\mathbb{N}-1\}\cup \{\pm (3^n,3^n,3^n,3^n):n\in 2\mathbb{N}\}.
\end{align*}
\end{proposition}
\begin{proof}
It was observed in the proof of Theorem~5.3 in \cite{BK:AutomCSC} that the map $\phi$ on the generators
\[
a\mapsto a^{3}, \quad b\mapsto b^{3}, \quad x\mapsto x^{-3}, \quad y\mapsto y^{-3}
\]
preserves the defining relations of the group $\Gamma_3$ and extends to an endomorphism of $\Gamma_3$ (actually, to an injective endomorphism, because $\Gamma_3$ is just-infinite). Therefore, the relation $ax=x^{-1}b$ produces the relations $a^{3^{n-1}}x^{-3^{n-1}}=x^{3^{n-1}}b^{3^{n-1}}$ for odd $n$ and $a^{3^n}x^{3^n}=x^{-3^n}b^{3^n}$ for even $n$. The corresponding tuples with negative sign in $P$ comes from the relation $a^{-1}x^{-1}=xb^{-1}$. Hence, the stated relations indeed hold in the group $\Gamma_3$.

Let us show that there are no other relations. The subgroups $A=\langle a,b\rangle$ and $B=\langle x,y\rangle$ are free of rank $2$, and the group $\Gamma_3$ admits two normal forms:
\[
\forall g\in \Gamma_3 \quad \exists! u,u'\in A, v,v'\in B: \quad g=uv=v'u', \quad \mbox{ here $|u|=|u'|, |v|=|v'|$}.
\]
The last equation defines a left action $A\curvearrowright B$ and a right action $A\curvearrowleft B$, denote $\pi_u(v)=v'$ and $\pi_v(u)=u'$. It is direct to check that the size of the $\pi_a$-orbit of $x^2$ and the size of the $\pi_x$-orbit of $a^2$ is $12$.

Assume $a^nx^m=x^{-m'}b^{n'}$ for some $n,m,n',m'\geq 1$. The normal form immediately implies that $n'=n$ and $m'=m$.
The relations $a^nx^m=x^{-m}b^{n}$ and $a^{9}x^{9}=x^{-9}b^{9}$ imply $\pi_a^n(x^2)=x^{-2}=\pi_a^{9}(x^2)$. Then $n-9$ is divisible by $12$ and $n$ is a multiple of $3$. Similarly $m$ is a multiple of $3$. Using the isomorphism $\phi$ between $\Gamma_3$ and $\langle a^3,b^3,x^3,y^3\rangle$, we get the relation $a^{\frac{n}{3}}x^{-\frac{m}{3}}=x^{\frac{m}{3}}b^{\frac{n}{3}}$. By iterating the argument, we obtain the statement.
\end{proof}

\begin{corollary}
The word problem in $\Gamma_3$ is not poly-context-free.
\end{corollary}
\begin{proof}
Let $\WP$ be the word problem of $\Gamma_3$ with respect to $a,b,x,y$. Then
\[
\WP\cap a^*x^*(b^{-1})^*x^*=\{\epsilon\}\cup\{ a^{9^n}x^{9^n}(b^{-1})^{9^n}x^{9^n} : n\geq 0 \},
\]
which is not poly-context-free.
\end{proof}

\section{The groups $\Gamma_{\tau}$ }

Let us define the groups $\Gamma_{\tau}$ introduced in \cite{StixVdovina:SimplyTrans}; the original notations are preserved.

Let $q$ be a power of an odd prime $p$, and $\bF_q$ be a field of order $q$. Fix a parameter $1\neq \tau\in\mathbb{F}_q^{*}$ and a non-square $c\in\mathbb{F}_q^{*}$.

Let $K=\mathbb{F}_q(t)$ be the rational function field over $\mathbb{F}_q$ and $D$ be the quaternion algebra over $K$ with $K$-basis $1,Z,F,ZF$ and multiplication
\begin{align*}
Z^2=c, \ F^2=t(t-1), \ ZF=-FZ.
\end{align*}
The algebra $D$ does not depend on the non-square $c$ up to isomorphism. The group $\Gamma_{\tau}$ will be defined as a finitely generated subgroup of the group $D^*/K^*$.

The subset $\bF_q[Z]\subset D$ is a field of order $q^2$. The norm map of the extension $\bF_q\subset \bF_q[Z]$ is the map $N:\bF_q[Z]^{*}\rightarrow\bF_q^{*}$, $N(\xi)=\xi\cdot\overline{\xi}=\xi^{q+1}$ for $\xi\in\bF_q[Z]$. Define the sets
\begin{align*}
N_c=\{\xi\in \bF_q[Z]^{*} : N(\xi)=-c\} \ \mbox{ and } \ M_{\tau}=\{\eta\in \bF_q[Z]^{*} : N(\eta)=\tfrac{c\tau}{1-\tau}\}.
\end{align*}
Now the group $\Gamma_\tau=\langle A,B_\tau\rangle$ is defined as the subgroup of the group $D^{*}/K^{*}$ generated by
\begin{align*}
A=\{ a_{\xi}:=ct+\xi FZ :\xi\in N_c \} \ \mbox{ and } \ B_\tau=\{ b_{\eta}:=ct+\eta FZ :\eta\in M_\tau \}.
\end{align*}

We need the following properties of $A,B_\tau$ and $\Gamma_\tau$ proved in \cite{StixVdovina:SimplyTrans}:
\begin{enumerate}
  \item The sets $A$ and $B_\tau$ are disjoint sets of size $q+1$ and are closed under taking inverses: $(tc+\xi FZ)^{-1}=tc-\xi FZ$.
  \item For every $(\xi,\eta)\in N_c\times M_{\tau}$ there exists a unique $(\mu,\lambda)\in N_c\times M_{\tau}$ such that $a_\xi b_\eta=b_\lambda a_\mu$, the pair $(\mu,\lambda)$ is a unique solution of the system
\begin{equation}\label{eqn:xi_eta_lam_mu_conditions}
\xi\eta^q=\lambda\mu^q \ (equiv., \mbox{$\xi\overline{\eta}=\lambda\overline{\mu}$}) \quad \mbox{ and } \quad \xi+\eta=\lambda+\mu.
\end{equation}

  \item The group $\Gamma_\tau$ is a torsion-free arithmetic lattice in $D^{*}/K^{*}$ with finite presentation
\[
\Gamma=\langle A,B_\tau\,|\, a_\xi a_{-\xi}=e, b_\eta b_{-\eta}=e, \mbox{ and } a_\xi b_\eta=b_\lambda a_\mu \ \mbox{ iff (\ref{eqn:xi_eta_lam_mu_conditions}) holds }\rangle.
\]
The subgroups $\langle A\rangle$ and $\langle B_\tau\rangle$ are free groups of rank $\frac{q+1}{2}$. The group $\Gamma_\tau$ admits an exact factorization by these subgroups; every element $g\in\Gamma_\tau$ admits two unique presentations in the normal forms:
\begin{equation}\label{eqn:normal_forms}
g=ab \ \mbox{ for $a\in\langle A\rangle$, $b\in\langle B_\tau\rangle$} \quad \mbox{ and } \quad g=cd \ \mbox{ for $c\in\langle B_\tau\rangle$, $d\in\langle A\rangle$}.
\end{equation}
\end{enumerate}

\begin{example}
Let $q=3$ and choose $\tau,c\in\bF_3$ as $\tau=c=-1$. Then
\begin{align*}
N_c=\{\pm 1,\pm Z\} \ &\mbox{ and } \ M_\tau=\{\pm 1\pm Z\},\\
A=\{-t\pm F,-t\pm ZF\}  \ &\mbox{ and } \  B_\tau=\{-t\pm F\pm ZF\}.
\end{align*} 
If we denote $a=-t+FZ$, $b=-t-F$, $x=-t+F+FZ$, $y=-t-F+FZ$, then $a,b,x,y$ satisfy exactly the defining relations of the group $\Gamma_3$ from Section~\ref{sect:the_smallest_example}.
\end{example}

For $f(t)\in K^*$ and $\xi\in\bF_q[Z]$ we denote $a_{\xi}(f(t)):=cf(t)+\xi FZ\in D^{*}/K^{*}$.

\begin{lemma}\label{lemma:a(f(t))_a^p}
\begin{enumerate}
  \item Elements $\xi,\mu\in N_c$ and $\eta,\lambda\in M_{\tau}$ satisfy the system~(\ref{eqn:xi_eta_lam_mu_conditions}) if and only if the relation $a_\xi(f(t)) b_\eta(f(t))=b_\lambda(f(t)) a_\mu(f(t))$ holds for every $f(t)\in K^*$.
  \item For $\xi\in \bF_q[Z]$ and $f(t)\in K^{*}$ we have $a_\xi(f(t))^{p^k}=a_{\xi'}(g(t))$, where
\[
\xi'=(-c)^{\frac{1-p^k}{2}}N(\xi)^{\frac{p^k-1}{2}}\xi \quad \mbox{ and } \quad g(t)=\tfrac{f(t)^{p^k}}{(t(t-1))^{\frac{p^k-1}{2}}}\in K^{*}, \qquad  k\geq 1.
\]
\end{enumerate}
\end{lemma}
\begin{proof}
The first item follows from the direct computations:
\begin{align*}
a_\xi(f(t)) b_\eta(f(t))&=c^2f(t)^2-c\xi\eta^qt(t-1)-c(\xi+\eta)f(t)ZF,\\
b_\lambda(f(t)) a_\mu(f(t))&=c^2f(t)^2-c\lambda\mu^qt(t-1)-c(\lambda+\mu)f(t)ZF.
\end{align*}

Using the fact that $K$ is the center of the algebra $D$, we compute:
\begin{align*}
(\xi FZ)^{p^k}&=(\xi FZ \xi FZ)^{\frac{{p^k}-1}{2}} \xi FZ=(-c\xi\overline{\xi}t(t-1))^{\frac{{p^k}-1}{2}} \xi FZ=\\
&=(-cN(\xi))^{\frac{{p^k}-1}{2}} (t(t-1))^{\frac{{p^k}-1}{2}} \xi FZ,\\
(cf(t)+\xi FZ)^{p^k}&=c^{p^k}f(t)^{p^k}+ (\xi FZ)^{p^k}=c^{p^k}f(t)^{p^k}+(-cN(\xi))^{\frac{{p^k}-1}{2}} (t(t-1))^{\frac{{p^k}-1}{2}} \xi FZ=\\
&=c\tfrac{f(t)^{p^k}}{(t(t-1))^{\frac{{p^k}-1}{2}}}+c^{-p^k+1}(-cN(\xi))^{\frac{{p^k}-1}{2}}\xi FZ \quad \mbox{ in $D^{*}/K^{*}$}.
\end{align*}
\end{proof}

\begin{definition}
Define the map $\sigma_k:\bF_q[Z]\rightarrow\bF_q[Z]$ by the rule $\sigma_k(\xi)=(-c)^{\frac{1-p^k}{2}}N(\xi)^{\frac{p^k-1}{2}}\xi$.
\end{definition}

\begin{lemma}
The map $\sigma_k$ acts on the sets $N_c$ and $M_\tau$ as follows:
\begin{enumerate}
  \item If $\xi\in N_c$ then $\sigma_k(\xi)=\xi$.
  \item If $\eta\in M_{\tau}$ then $\sigma_k(\eta)=(\frac{\tau}{\tau-1})^{\frac{p^k-1}{2}}\eta$ and $\sigma_k(\eta)\in M_{\tau^{p^k}}$.
\end{enumerate}
\end{lemma}
\begin{proof}
The statements follow from the following computations:
\begin{enumerate}
  \item[1)] If $N(\xi)=-c$ then $(-c)^{\frac{1-p^k}{2}}N(\xi)^{\frac{p^k-1}{2}}=1$.
  \item[2)] If $N(\eta)=\frac{c\tau}{1-\tau}$ then $(-c)^{\frac{1-p^k}{2}}N(\eta)^{\frac{p^k-1}{2}}=(-c)^{\frac{1-p^k}{2}}(\tfrac{c\tau}{1-\tau})^{\frac{p^k-1}{2}}=(\tfrac{\tau}{\tau-1})^{\frac{p^k-1}{2}}$ \ and
\begin{align*}
N(\sigma_k(\eta))&=(\tfrac{\tau}{\tau-1})^{p^k-1}N(\eta)=(\tfrac{\tau}{1-\tau})^{p^k-1}\tfrac{c\tau}{1-\tau}=c(\tfrac{\tau}{1-\tau})^{p^k}=c\tfrac{\tau^{p^k}}{1-\tau^{p^k}}.
  \end{align*}
\end{enumerate}
\end{proof}

\begin{definition}
Define $k_{\tau}$ as the smallest $k\in\mathbb{N}$ such that $(\frac{\tau}{\tau-1})^{\frac{p^k-1}{2}}=1$.
\end{definition}

Note that $\tau^{p^{k_{\tau}}}=\tau$ and $p^{k_{\tau}}\leq q^2$.

The next statement gives the key property.

\begin{proposition}\label{prop:endomorphism_phi_k}
The map $\phi_k$ defined on the generators $A,B_{\tau^{p^k}}$ by
\[
a_\xi\mapsto a_\xi^{p^k} \ \mbox{ for $\xi\in N_c$} \ \mbox{ and } \ b_{\sigma_{k}(\eta)}\mapsto b_{\eta}^{p^k} \ \mbox{ for $\eta\in M_\tau$}.
\]
extends to an injective homomorphism from $\Gamma_{\tau^{p^k}}$ to $\Gamma_{\tau}$. In particular, for $k=k_{\tau}$ the map $\phi_k$ is an injective endomorphism of $\Gamma_{\tau}$.
\end{proposition}
\begin{proof}
The map $\phi_k$ preserves the defining relations of the group $\Gamma_{\tau^{p^k}}$ by Lemma~\ref{lemma:a(f(t))_a^p} and extends to a homomorphism from $\Gamma_{\tau^{p^k}}$ to $\Gamma_\tau$. Since all non-trivial normal subgroups of the groups $\Gamma_\tau$ have finite index (see Proposition~41 in \cite{StixVdovina:SimplyTrans}), and the image of $\phi_k$ is infinite, $\phi_k$ is injective. In the case $k=k_{\tau}$ the map $\sigma_k$ acts identically on $N_c$ and $M_{\tau}$.
\end{proof}

\begin{corollary}\label{cor:relations_p^n}
Let $a_\xi b_\eta=b_\lambda a_\mu$ for different $\xi,\mu\in N_c$ and $\eta,\lambda\in M_{\tau}$. Then $a_\xi^{p^{n}}b_\eta^{p^{n}}=b_\lambda^{p^{n}}a_\mu^{p^{n}}$ if and only if $n$ is a multiple of $k_{\tau}$.
\end{corollary}
\begin{proof}
We can get the stated relations by applying the endomorphism $\phi_{k_{\tau}}$; however, there are more direct arguments. By Lemma~\ref{lemma:a(f(t))_a^p} the relation $a_\xi^{p^{n}}b_\eta^{p^{n}}=b_\lambda^{p^{n}}a_\mu^{p^{n}}$ holds if and only if $\xi, \sigma_n(\eta)$ and $\sigma_n(\lambda),\mu$ satisfy the conditions (\ref{eqn:xi_eta_lam_mu_conditions}). The second condition gives:
\[
\xi+\sigma_n(\eta)=\sigma_n(\lambda)+\mu \quad \Rightarrow \quad \xi-\mu=\sigma_n(\lambda)-\sigma_n(\eta)=(\tfrac{\tau}{\tau-1})^{\frac{p^n-1}{2}}(\lambda-\eta).
\]
Since $\xi+\eta=\lambda+\eta$ and $\xi\neq\mu$, $\lambda\neq\eta$, we get $(\tfrac{\tau}{\tau-1})^{\frac{p^n-1}{2}}=1$ and $n$ is divisible by $k_{\tau}$.
\end{proof}

Further, we will show that for non-commuting $a,b,c,d$ the relation $a^{*}b^{*}=c^{*}d^{*}$ is possible only for $p$th powers. To do this, we need the notion of an anti-torus.

\begin{definition}
An \textit{anti-torus} in the group $\Gamma_\tau$ is a subgroup $\langle a,b\rangle$ for $a\in \langle A\rangle$, $b\in \langle B_\tau\rangle$ that do not have commuting non-trivial powers, i.e., $a^nb^m\neq b^ma^n$ for all $n,m\in\mathbb{Z}\setminus\{0\}$.
\end{definition}

An open conjecture by Wise \cite[Con.~4.10]{Wise:Diss} states that every irreducible lattice in the product of trees contains an anti-torus. This was confirmed by Rattaggi \cite{Rat:Anti-tori} for the lattices $\Gamma_{p,q}$ defined by Mozes \cite{Mozes}. The proof relates anti-tori in $\Gamma_{p,q}$ and non-commuting Hamilton quaternions in $\mathbb{H}(\mathbb{Z})$, and directly extends to the groups $\Gamma_{\tau}$ as follows. A group $G$ is called commutative transitive, if the commutativity relation is transitive on the set $G\setminus\{e\}$, that is, $g_1g_2=g_2g_1$, $g_2g_3=g_3g_2$ implies $g_1g_3=g_3g_1$ for $g_1,g_2,g_3\neq e$.

\begin{proposition}
The group $D^{*}/K^{*}$ is commutative transitive.
\end{proposition}
\begin{proof}
The statement holds for any quaternion algebra over a field of characteristic $\neq 2$. Two elements $x,y\in D^{*}\setminus K^*$ commute if and only if the imaginary parts of $x$ and $y$ are $K^*$-proportional, and this is equivalent to the condition that the images of $x$ and $y$ in $D^{*}/K^*$ commute. It follows that if $xy=yx$ and $yz=zy$ for $x,y,z\in D^{*}\setminus K^*$, then the imaginary parts of $x,y,z$ are $K^*$-proportional and $xz=zx$.
\end{proof}

\begin{corollary}\label{cor:G_tau_anti-torus}
The group $\Gamma_\tau$ is commutative transitive and has an anti-torus. Moreover, for $a\in \langle A\rangle$, $b\in \langle B_\tau\rangle$, the subgroup $\langle a,b\rangle$ is an anti-torus in $\Gamma_\tau$ if and only if $a$ and $b$ do not commute.
\end{corollary}
\begin{proof}
The group $\Gamma_\tau$ is commutative transitive as a subgroup of $D^{*}/K^{*}$, and we can apply Corollaries~6, 7 from \cite{Rat:Anti-tori}.
\end{proof}

\begin{remark}
If $a\in A$ and $b\in B_\tau$ commute, then the language $L=\WP(\Gamma_\tau)\cap a^{*}b^{*}(a^{-1})^{*}(b^{-1})^{*}$ is the intersection of two context-free languages $a^{n}b^{*}(a^{-1})^{n}(b^{-1})^{*}$ and $a^{*}b^{m}(a^{-1})^{*}(b^{-1})^{m}$, here $P(L)=\{(n,m,n,m) : n,m\in\mathbb{N}_0\}$ is semilinear. This situation can indeed happen for the group $\Gamma_\tau$. For example,
for $q=5$ and $c=2,\tau=3$, we have
\begin{align*}
A=\{2t\pm 2F, 2t\pm F\pm FZ\} \ \mbox{ and } \ B_\tau=\{2t\pm F, 2t\pm 2F\pm 2FZ\},
\end{align*}
where $2t\pm 2F,2t\pm F$ and $2t\pm (F+FZ), 2t\pm 2(F+FZ)$ commute.
\end{remark}

Let us show that certain relations in $\Gamma_\tau$ are possible only for commuting elements.

\begin{lemma}\label{lemma:different_xi_mu}
Let $a_\xi b_\eta=b_\lambda a_\mu$ for $\xi,\mu\in N_c$ and $\eta,\lambda\in M_{\tau}$. Then
\begin{enumerate}
  \item[1)] $\lambda=\eta$ if and only if $\mu=\xi$;
  \item[2)] $\lambda=\overline{\eta}$ if and only if $\mu=\overline{\xi}$.
\end{enumerate}
\end{lemma}
\begin{proof}
Immediately follows from the equations (\ref{eqn:xi_eta_lam_mu_conditions}).
\end{proof}

\begin{lemma}\label{lemma:lambda_neq_eta_conj_eta}
Let $\xi,\mu,\chi\in N_c$ and $\eta,\lambda\in M_{\tau}$ satisfy the equations
$a_\xi b_\eta =b_\lambda a_\mu$ and $a_\mu b_\eta =b_\lambda a_\chi$.
Then $\lambda=\eta$ and $\xi=\mu=\chi$.
\end{lemma}
\begin{proof}
Let us write the conditions (\ref{eqn:xi_eta_lam_mu_conditions}) for the given relations:
\begin{align*}
\left\{
\begin{array}{l}
\xi+\eta=\lambda+\mu\\
\mu+\eta=\lambda+\chi\\
\xi\overline{\eta}=\lambda\overline{\mu}\\
\mu\overline{\eta}=\lambda\overline{\chi}
\end{array}
\right.
\qquad \Rightarrow \qquad
\begin{array}{l}
\qquad \mu=\xi+\eta-\lambda\\
\qquad \chi=\mu+\eta-\lambda=\xi+2\eta-2\lambda\\
\left\{
\begin{array}{l}
\xi\overline{\eta}=\lambda(\overline{\xi}+\overline{\eta}-\overline{\lambda})\\
(\xi+\eta-\lambda)\overline{\eta}=\lambda (\overline{\xi}+2\overline{\eta}-2\overline{\lambda})
\end{array}
\right.
\end{array}.
\end{align*}
The last system leads to $\lambda=\eta$.
\end{proof}

\begin{lemma}\label{lemma:n=1_or_m=1}
Let $ab=cd$ for $a,d\in A$ and $b,c\in B_\tau$. If $ab^n=c^nd$ or $a^nb=cd^n$ for some $n>1$, then $a=d$ and $b=c$.
\end{lemma}
\begin{proof}
We have the following relations in the group $\Gamma_{\tau}$:
\begin{align*}
ab^n=cdb^{n-1}=c^nd \ &\Rightarrow \ db^{n-1}=c^{n-1}d \ \Rightarrow \ db^{(n-1)k}=c^{(n-1)k}d \ \mbox{ for all $k\geq 1$}.
\end{align*}
By applying the endomorphism $\phi_{k_{\tau}}$, put $m=p^{k_{\tau}}$, we get the relations:
\begin{align*}
d^{m}b^{m(n-1)k}=c^{m(n-1)k}d^{m} \ \Rightarrow \ d^{m-1}b^{m(n-1)k}=b^{m(n-1)k}d^{m-1}.
\end{align*}
Thus, $\langle d,b\rangle$ is not an anti-torus, $d$ and $b$ commute by Corollary~\ref{cor:G_tau_anti-torus}, and $a=d, b=c$ by the normal form~(\ref{eqn:normal_forms}).
\end{proof}

We are ready to prove the main result.

\begin{theorem}
Let $\WP\subset (A\cup B_\tau)^{*}$ be the word problem of the group $\Gamma_{\tau}$. Let $a,c\in A$, $b,d\in B_\tau$ be such that $abcd=e$ and $a,b$ do not commute. Then
\[
P(\WP\cap a^*b^*c^*d^*)=\{(0,0,0,0)\}\cup \{ (p^{k_{\tau} n},p^{k_{\tau}n},p^{k_{\tau}n},p^{k_{\tau}n}) : n\geq 0\}.
\]
\end{theorem}
\begin{proof}
The relation $a^mb^mc^md^m=e$ holds for $m=p^{k_{\tau}n}$ by Corollary~\ref{cor:relations_p^n}. We will show that no other $m$ satisfies the equation. Note that $a\neq c^{-1}$ and $b\neq d^{-1}$ by Lemma~\ref{lemma:different_xi_mu}.

The uniqueness of the normal form (\ref{eqn:normal_forms}) defines a left action of the group $\langle A\rangle$ on $\langle B_\tau\rangle$ and a right action of $\langle B_\tau\rangle$ on $\langle A\rangle$. Namely, for $g\in \langle A\rangle$ and $h\in\langle B_\tau\rangle$ define the permutations $\pi_g\in Sym(\langle B_\tau\rangle)$ and $\pi_h\in Sym(\langle A\rangle)$ by the rule $gh=\pi_g(h)\pi_h(g)$. Note that $\pi_g$ and $\pi_h$ preserve the length and prefixes of reduced words; if $\pi_g(h_1h_2\ldots h_m)=h'_1h'_2\ldots h'_m$ then $\pi_g(h_1h_2\ldots h_k)=h'_1h'_2\ldots h'_k$ for every $k\leq m$. (Actually, the $\pi_g$ and $\pi_h$ are automaton transformations, see \cite{BK:AutomCSC}.)
For our relation $ab=d^{-1}c^{-1}$ we have $\pi_a(b)=d^{-1}$ and $\pi_b(a)=c^{-1}$.

We will show that the size of the $\pi_a$-orbit of $b^2$ is divisible by $p$, denote this size by $L$. Let $l$ be the size of the $\pi_a$-orbit of $b$, here $l>1$ and $l$ divides $L$. Then $a^lb=bd'$ for some $d'\in \langle A\rangle$. By applying the endomorphism $\phi_{k_{\tau}}$ from Proposition~\ref{prop:endomorphism_phi_k} to this relation, we get
\[
a^{p^{k_{\tau}}l}b^{p^{k_{\tau}}}=b^{p^{k_{\tau}}}d'' \quad \mbox{ for some  $d''\in \langle A\rangle$}.
\]
Therefore, $\pi_a^{p^{k_{\tau}}l}(b^2)=b^2$, and $L$ divides $p^{k_{\tau}}l$.
The first letter of $\pi_a^i(b^2)$ for $i=0,1,\ldots,l$ is equal to $d^{-1}$ only for $i=1$.
By Lemma~\ref{lemma:lambda_neq_eta_conj_eta}, $\pi_{c^{-1}}(b)\neq d^{-1}$, and therefore $\pi_a(b^2)\neq d^{-2}$. It follows that $\pi_a^i(b^2)\neq d^{-2}$ for $i=0,1,\ldots,l$.
At the same time, the relation $a^{p^{k_{\tau}}}b^{p^{k_{\tau}}}=d^{-p^{k_{\tau}}}c^{-p^{k_{\tau}}}$ implies $\pi_a^{p^{k_{\tau}}}(b^2)=d^{-2}$, and therefore $d^{-2}$ belongs to the $\pi_a$-orbit of $b^2$. Hence $L>l$, $L/l$ divides $p^{k_{\tau}}$ and $L$ is divisible by $p$. (Also $\pi_a^{p^{k_{\tau}}}(b)=d^{-1}=\pi_a(b)$, and therefore $l$ divides $p^{k_{\tau}}-1$ and is coprime with $p$.) The same arguments are applicable to $\pi_b$, and the size of the $\pi_b$-orbit of $a^2$ is divisible by $p$.

Let us return to the relations of the form $a^nb^mc^{n'}d^{m'}=e$ for $n,m,n',m'\geq 0$. The normal form (\ref{eqn:normal_forms}) implies that $m'=m$ and $n'=n$, so we are left with $a^nb^mc^nd^m=e$. If $n=0$ or $m=0$, then $n=m=0$. If $n=1$ or $m=1$, then $n=m=1$ by Lemma~\ref{lemma:n=1_or_m=1}. Consider the case $n,m\geq 2$. We have $a^{p^{k_\tau}}b^{p^{k_\tau}}=d^{-p^{k_\tau}}c^{-p^{k_\tau}}$ and $\pi_a^{p^{k_\tau}}(b^2)=d^{-2}=\pi_a^n(b^2)$. It follows that $n-p^{k_\tau}$ is divisible by $L$, and hence $n$ is divisible by $p$. Similarly $m$ is divisible by $p$. We can repeat the same arguments with $a,b,c,d$ replaced by their $p$-powers using Lemma~\ref{lemma:a(f(t))_a^p}. Note that the non-commuting condition is preserved under such transition. It follows that $n=m$ is a power of $p$, and we can apply Corollary~\ref{cor:relations_p^n}.
\end{proof}

\begin{corollary}\label{cor:not_poly_context_free}
The word problem in $\Gamma_\tau$ is not poly-context-free.
\end{corollary}

Corollary~\ref{cor:not_poly_context_free} is possibly generalizable to all irreducible lattices in the product of trees. Recall that such a lattice is called reducible if it is virtually a direct product of free groups, and therefore it is poly-context-free.

\begin{conjecture}
A lattice in the product of trees is poly-context-free if and only if it is reducible.
\end{conjecture}

The approach above does not work for all irreducible lattices as the following examples show. Moreover, if the aforementioned conjecture of Wise does not hold, and there exists an irreducible lattice $\Gamma=\langle A,B\rangle$ without anti-tori, then $\WP_\Gamma\cap a^*b^*c^*d^*$ is poly-context free for all $a,c\in\langle A\rangle$, $b,d\in\langle B\rangle$.

\begin{example}\label{ex:group G_4}
There are two irreducible lattices in $T_4\times T_4$ with the smallest number of squares. One of them is the group $\Gamma_3$, and the other one has presentation
\[
\Gamma_4=\langle a,b,x,y\,|\,ax=yb, ay=y^{-1}b, bx=xa^{-1}, by=x^{-1}a\rangle.
\]
The map $\phi(a)=a^4$, $\phi(b)=b^4$ and $\phi(x)=x$, $\phi(y)=y$ preserves the defining relations and extends to an endomorphism of $\Gamma_4$. This observation was used in \cite{BK:AutomCSC} to prove that $\Gamma_4$ is not residually finite. It can also be used to describe the intersection $\WP\cap a^*b^*c^*d^*$ for the defining relations:
\begin{align*}
P(\WP\cap a^*x^*(b^{-1})^*(y^{-1})^{*})&=P(\WP\cap a^*y^*(b^{-1})^*y^{*})=P(\WP\cap b^*y^*(a^{-1})^*x^{*})=\\
&=\{(0,0,0,0)\}\cup\{ (1+3n,1,1+3n,1) : n\geq 0\},\\
P(\WP\cap b^*x^*a^*(x^{-1})^{*})&=\{(0,0,0,0)\}\cup\{ (n,1,n,1), (0,n,0,n) : n\geq 0\}\cup\\
 &\quad\cup \{ (3n,1+4m,3n,1+4m) : n,m\geq 0\}.
\end{align*}
The first three languages are context-free, and the last one is the intersection of two context-free languages. 
\end{example}

\begin{example}\label{ex:group Wise}
The next group was used by Wise \cite{Wise:Diss} to construct the first example of a non-residually finite group in the classes of small cancellation groups, automatic groups, and groups acting geometrically on CAT(0)-spaces:
\[
\Gamma_{3,2}=\langle a,b,c,x,y\,|\,ax=xb, ay=yb, bx=ya, by=xc, cx=yc, cy=xa\rangle.
\]
The intersection $\WP\cap a^*b^*c^*d^*$ for the defining relations are context-free languages:
\begin{align*}
P(\WP\cap a^*x^*(b^{-1})^*(x^{-1})^*)&=P(\WP\cap a^*y^*(b^{-1})^*(y^{-1})^*)=\\
&=\{(0,n,0,n), (n,1,n,1) : n\geq 0\},\\
P(\WP\cap b^*x^*(a^{-1})^*(y^{-1})^*)&=\{(n,n,n,n) : n=0,1,3\},\\
P(\WP\cap b^*y^*(c^{-1})^*(x^{-1})^*)&=\{(n,n,n,n) : n=0,1\},\\
P(\WP\cap c^*x^*(c^{-1})^*(y^{-1})^*)&=\{(1,n,1,n), (n,0,n,0) : n\geq 0\},\\
P(\WP\cap c^*y^*(a^{-1})^*(x^{-1})^*)&=\{(n,n,n,n) : n=0,1\}.
\end{align*}
\end{example}


\end{document}